\documentclass[12pt]{article}
\usepackage{verbatim}
\usepackage{amssymb,amsmath}
\usepackage{amsthm}
\usepackage{graphicx}
\usepackage{epsfig}
\usepackage{xcolor}
\usepackage{tikz}
\usetikzlibrary{arrows.meta,calc}
\newtheorem{theorem}{Theorem}
\newtheorem{corollary}{Corollary}[section]
\newtheorem{lemma}[corollary]{Lemma}
\newtheorem{proposition}[corollary]{Proposition}

\newcommand{\beq}{\begin{eqnarray*}}
\newcommand{\eeq}{\end{eqnarray*}}
\newcommand{\beqn}{\begin{eqnarray}}
\newcommand{\eeqn}{\end{eqnarray}}

\newcommand{\Pro}{{\bf P}}
\newcommand{\Z}{{\mathbb Z}}

\newcommand{\C}{{\mathbb C}}

\newcommand{\x}{{\bf x}}
\newcommand{\y}{{\bf y}}

\def \ga {{\gamma}}

\def \p {\partial}

\def \Half {{\mathbb H}}

\newenvironment{remark}[1][Remark]{\begin{trivlist}
\item[\hskip \labelsep {\bfseries #1}]}{\end{trivlist}}
\newenvironment{definition}[1][Definition]{\begin{trivlist}
\item[\hskip \labelsep {\bfseries #1}]}{\end{trivlist}}

\newcommand{\onearm}{\pi_1^+}
\newcommand{\hp}{\widehat p}

\begin{document}
\title{An Exact Strip Observable and the Half-Plane One-Arm Probability\\
for Critical Bond Percolation on the Square Lattice}

\author{Wang Zhou\thanks{ Department of Statistics and Data Science, National University of Singapore, Singapore 117546. Email: stazw@nus.edu.sg}\\
{\it \small National University of Singapore}}

\date{}

\maketitle

\begin{abstract}
We prove the half-plane one-arm exponent for critical Bernoulli bond percolation on the square lattice.  Ikhlef and Ponsaing obtained an exact formula for the probability that the unique infinite percolation hull running along an odd-width strip passes through a prescribed boundary edge.  Their interpolation, however, treats a rational function as a polynomial.  We clear the full denominator, prove a coordinatewise Laurent-degree bound directly at $q^3=1$, and establish the exact strip formula from the deletion relations.  At the homogeneous point, the qKZ expression is the Bernoulli probability of this boundary-passage event.  A finite local modification, the RSW theorem, and one-arm quasi-multiplicativity then give, uniformly for $1\le r<R$,
$\Pro_{1/2}(A_1^+(r,R))\asymp (r/R)^{1/3}$.
\end{abstract}


\tableofcontents

\section{Introduction}

Boundary arm probabilities are basic quantities in planar percolation.  They
control how critical clusters approach a boundary and how exploration paths
visit small boundary neighbourhoods.  Russo--Seymour--Welsh (RSW) estimates and
quasi-multiplicativity compare such probabilities across scales.  These tools
do not by themselves determine the exact exponent.

We consider critical Bernoulli bond percolation on the square lattice.  Fix
integers $1\le r<R$.  Let $A_1^+(r,R)$ be the event that an open path in a
half-plane joins two concentric box boundaries of radii $r$ and $R$.  Our main
result is the following.

\begin{theorem}\label{thm:main-intro}
There are universal constants $0<c<C<\infty$ such that, for all integers
$1\le r<R$,
\[
 c\left(\frac rR\right)^{1/3}
 \le \Pro_{1/2}\bigl(A_1^+(r,R)\bigr)
 \le C\left(\frac rR\right)^{1/3}.
\]
\end{theorem}

Ikhlef and Ponsaing studied a related passage event in an odd-width strip
\cite{IP}.  Their method uses the reflecting-boundary qKZ equations and
classical character identities.  It leads to an exact expression for the
probability that the infinite hull passes through a prescribed reflecting
edge.  One step in their argument needs correction.  The inhomogeneous
quantity is a rational function of the spectral parameters, whereas the
interpolation in \cite[proof of Proposition~4.6]{IP} treats it as a
polynomial.

We repair this step by clearing the complete denominator.  We also include the
square of the qKZ normalization.  The resulting numerator is a Laurent
polynomial.  We prove a coordinatewise degree bound directly from the qKZ
component relation at $q^3=1$.  The deletion relations then provide enough
specializations to determine the numerator by Laurent interpolation.

The exact algebraic formula must still be connected with Bernoulli
percolation.  At the homogeneous point, the double-row transfer matrix defines
a finite Markov chain on odd link patterns.  We prove that this chain is
irreducible and aperiodic.  Its stationary law is the normalized homogeneous
qKZ vector.  The corresponding matrix element is therefore the probability
that the infinite hull passes through the reflecting edge $e_b$.

We next compare this strip event with an ordinary open connection across the
strip.  The comparison changes only finitely many bonds near $e_b$ and loses
at most a fixed factor in probability.  The RSW theorem then compares the strip
connection with the half-plane one-arm event.  Finally,
quasi-multiplicativity gives the estimate for all pairs of scales $r<R$.

The paper is organized as follows.  Section~2 defines the percolation events
and the odd link-pattern space.  Section~3 proves the inhomogeneous strip
formula.  Section~4 evaluates its homogeneous specialization and obtains an
exact recurrence in the strip width.  Section~5 compares the event
$\{e_b\subseteq\ga_L\}$ with an open strip connection.  Section~6 compares
that connection with the half-plane one-arm event and proves the theorem.

We use the standard qKZ exchange, reflection, and vector-deletion relations
from \cite{IP}.  We also use the classical character evaluations stated in
Section~4.  The identities for the event $\{e_b\subseteq\ga_L\}$, the degree
bound, the corrected interpolation, and the Bernoulli interpretation are
proved here.

\section{Percolation events and strip connectivities}

We work with critical Bernoulli bond percolation on the square lattice.  Let
\[
  \Half=\{(x,y)\in\Z^2:x\ge 0\}.
\]
For integers $1\le r<R$, let $A_1^+(r,R)$ be the following event.  There is
a primal-open path in $\Half$ from the boundary of a box of radius $r$ to the
boundary of the concentric box of radius $R$.  Both boxes are centred at the
origin.  We write
\[
  \onearm(r,R)=\Pro_{1/2}\bigl(A_1^+(r,R)\bigr),
  \qquad
  \onearm(R)=\onearm(1,R).
\]
The choice between Euclidean and $\ell^\infty$ boxes is immaterial up to a
universal constant.  The same is true if the initial boundary vertex is moved
by a bounded distance.

For odd $L$, consider the infinite strip of Ikhlef and Ponsaing
\cite[Fig.~2(b) and Definition~4.2]{IP}.  One primal side is free.  The
opposite side is wired.  All interior primal bonds are
independent Bernoulli variables of parameter $1/2$.  In the medial representation, the configuration consists of nonintersecting loops and one unpaired interface. The unpaired interface is present because a transverse section contains an odd number of medial endpoints. It extends infinitely in both longitudinal directions and is the unique bi-infinite percolation interface running along the strip. We denote it by \(\ga_L\).

Equivalently, cut the strip at a transverse section.  The connectivity below
the cut is an \emph{odd link pattern}.  It is a noncrossing matching of the
$L$ endpoints, with exactly one endpoint connected to infinity.  We denote the
finite set of these patterns by $\mathrm{LP}_L$.  Gluing an upward and a
downward odd link pattern produces closed loops and exactly one through-going
interlace.  When the strip is extended in both directions, this interlace becomes
$\ga_L$.  This finite set will be used in the probabilistic identification
in Proposition~\ref{prop:stochastic-identification}.

\begin{definition}
In the central double-row tangle, the lower auxiliary line carries the
parameter $w$ and is directed from left to right.  The upper line carries
$w^{-1}$ and is directed from right to left.  Their left endpoints are
joined by one curved medial boundary edge.  This is the left reflector
bracketed by the two marking dots in
\cite[Fig.~2(b) and Definition~4.2]{IP}, which is denoted by $e_b$.
It is an \emph{unoriented medial edge}, not one of the $L$ sites of a
transverse link pattern.

For every completed two-row loop diagram, let $\mathcal C_b$ be the interlace
containing $e_b$.  The operator $\widehat\rho_b^{(L)}$ from
\cite[Definition~4.2]{IP} restricts to configurations for which
$\mathcal C_b$ is the through-going interlace.  In the infinite Bernoulli strip,
this is the event that $e_b$ is contained in $\ga_L$.  We set
\[
  H_L=\{e_b\subseteq \ga_L\},
  \qquad
  \hp_L=\Pro_{1/2}(H_L).
\]
No orientation of $e_b$ enters either the event or the algebraic observable.
Figure~\ref{fig:reflector-convention} shows the convention used below.
\end{definition}

\begin{figure}[ht]
\centering
\begin{tikzpicture}[scale=.86,>=Latex]
  \draw[rounded corners,thick] (0,0) rectangle (6,2.2);
  \draw[thick,->] (0,.55)--(5.85,.55);
  \draw[thick,<-] (0,1.65)--(5.85,1.65);
  \draw[thick] (6,.55) .. controls (6.75,.55) and (6.75,1.65) .. (6,1.65);
  \draw[red!70!black,very thick]
      (0,.55) .. controls (-.85,.55) and (-.85,1.65) .. (0,1.65);
  \fill (-1.06,1.10) circle (2.2pt);
  \fill (.18,1.10) circle (2.2pt);
  \node[red!70!black,anchor=east] at (-1.18,1.43) {$e_b$};
  \foreach \x/\lab in {1/$z_1$,2/$z_2$,4.7/$z_{L-1}$,5.7/$z_L$}{
     \draw[thin] (\x,0)--(\x,2.2);
     \node[anchor=north] at (\x,-.04) {\lab};
  }
  \node[anchor=south] at (3,.58) {$w$};
  \node[anchor=north] at (3,1.62) {$w^{-1}$};
  \node[align=center,anchor=north] at (3,-.65)
  {lower line: left to right\qquad upper line: right to left};
\end{tikzpicture}
\caption{Convention corresponding to $\widehat\rho_b^{(L)}$.  The two black
dots bracket the curved left reflecting medial edge $e_b$.  The event is that
the edge $e_b$ is contained in the hull $\ga_L$.  The edge is not
oriented.}
\label{fig:reflector-convention}
\end{figure}
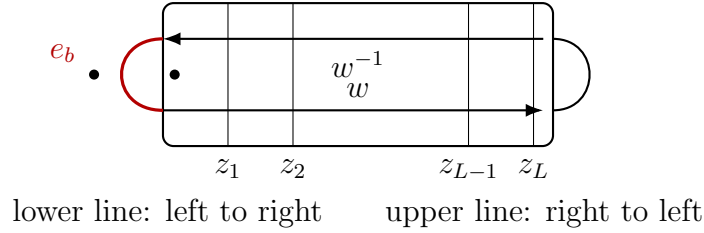

The same law is obtained from finite longitudinal strips.  Fix odd link
patterns at the two remote ends and send both lengths to infinity.
Proposition~\ref{prop:stochastic-identification} proves that the limit does not
depend on the chosen patterns.

\section{The inhomogeneous strip formula}\label{sec:qkz}

We now complete the interpolation argument for the boundary-contact
observable.  The important point is that the probability itself is a rational
function of the spectral parameters.  Interpolation must therefore be applied
after all denominators have been cleared.

We use the ordinary reflecting-boundary qKZ relations from
\cite[Sections~3.3--3.5]{IP}.  These are the exchange, reflection, and
vector-deletion relations.  We also use the character identities stated in
Lemma~\ref{lem:character-specializations}.

We follow the Dyck-path notation of de Gier and Pyatov
\cite[Figure~2]{dGP}.  Their generic semisimplicity assumption is not needed.
Lemma~\ref{lem:root-component-relation} derives the required triangular
relation directly from the link-pattern qKZ equation.

Before the root-of-unity specialization, the qKZ identities are rational
identities.  We first clear every denominator.  The resulting
Laurent-polynomial identities may then be specialized to $q=e^{2\pi i/3}$.
All later divisions are understood as identities of rational functions.  We will prove the normalization $Z_L=\chi_L$, the Bernoulli interpretation,
the local relations for $\widehat\rho_b^{(L)}$, the degree bounds, and the
corrected interpolation.

Set
\[
  q=e^{2\pi i/3},
  \qquad [z]=z-z^{-1},
\]
and
\begin{equation}\label{eq:def-k}
  k(a,b)=[qb/a][q/(ab)]
  =q^2a^{-2}+q^{-2}a^2-b^2-b^{-2}.
\end{equation}
Let $L=2m-1$ be odd.  Let $z_1,\ldots,z_L$ be the vertical spectral
parameters, and let $w$ be the horizontal parameter.  Put
\[
  x_j=z_j^2.
\]
A Laurent polynomial in $x_1,\ldots,x_L$ is a finite sum
\[
  \sum_{\boldsymbol a\in\Z^L}
  c_{\boldsymbol a}x_1^{a_1}\cdots x_L^{a_L}.
\]
Only finitely many coefficients are nonzero.  A partition of length at most
$M$ is a weakly decreasing sequence
$\lambda_1\ge\cdots\ge\lambda_M\ge0$.  We write
\begin{equation}\label{eq:def-symplectic-character}
 \operatorname{sp}_{\lambda}(y_1,\ldots,y_M)
 =\frac{
 \det\!\left(y_i^{\lambda_j+M-j+1}
             -y_i^{-\lambda_j-M+j-1}\right)_{i,j=1}^M}
 {\det\!\left(y_i^{M-j+1}-y_i^{-M+j-1}\right)_{i,j=1}^M}.
\end{equation}
The quotient is a symmetric Laurent polynomial.  For every $M\ge1$, set
\[
 \lambda_j^{(M)}=\left\lfloor\frac{M-j}{2}\right\rfloor,
 \qquad 1\le j\le M,
 \qquad
 \chi_M=\operatorname{sp}_{\lambda^{(M)}}.
\]
This fixes the character notation used throughout the paper.

Let
\begin{equation}\label{eq:downward-vector-definition}
 |\Psi_L^\downarrow(\boldsymbol z)\rangle
 =|\Psi_L(\boldsymbol z)\rangle
 :=\sum_{\alpha\in\mathrm{LP}_L}
   \psi_\alpha(z_1,\ldots,z_L)|\alpha\rangle.
\end{equation}
This is the primitive Laurent-polynomial qKZ solution.  It belongs to the
Temperley--Lieb system with reflecting boundaries.  We use the normalization
of Sections~3.4--3.6 in \cite{IP}.  Let $\bar\alpha$ be the rotation of $\alpha$
through $\pi$.  Following \cite[Section~4.1]{IP}, define
\begin{equation}\label{eq:upward-vector-definition}
 \langle\Psi_L^\uparrow(\boldsymbol z)|
 :=\sum_{\alpha\in\mathrm{LP}_L}
   \psi_\alpha(z_L,\ldots,z_1)\langle\bar\alpha|.
\end{equation}
Write
\begin{align*}
 Z_L^\downarrow(\boldsymbol x)
 &:=\sum_\alpha\psi_\alpha(z_1,\ldots,z_L)=:Z_L(\boldsymbol x),\\
 Z_L^\uparrow(\boldsymbol x)
 &:=\sum_\alpha\psi_\alpha(z_L,\ldots,z_1)
   =Z_L(x_L,\ldots,x_1).
\end{align*}
At loop weight $n=-(q+q^{-1})=1$, the diagrammatic pairing of any upward and
downward odd link patterns equals one.  Consequently the two-sided
normalization is
\begin{equation}\label{eq:def-two-sided-normalization}
 \mathcal Z_L(\boldsymbol x)
 :=\langle\Psi_L^\uparrow(\boldsymbol z)
       |\Psi_L^\downarrow(\boldsymbol z)\rangle
 =Z_L^\uparrow(\boldsymbol x)Z_L^\downarrow(\boldsymbol x).
\end{equation}

The identity $Z_L=\chi_L$ will be proved in
Proposition~\ref{prop:normalization}.
Let $\widehat\rho_b^{(L)}(w;\boldsymbol z)$ be the two-row operator from
\cite[Definition~4.2]{IP} obtained by restricting to configurations in which
$e_b$ lies on the through-going interlace.  Its normalized boundary-contact
weight is
\begin{equation}\label{eq:def-P-hat}
  \widehat P_L(w;\boldsymbol z)
  =
  \frac{\langle\Psi_L^\uparrow,
  \widehat\rho_b^{(L)}(w;\boldsymbol z)\Psi_L^\downarrow\rangle}
  {\mathcal Z_L(\boldsymbol x)}.
\end{equation}

For link patterns $\alpha$ and $\beta$, let
$\mathcal C_L(\beta,\alpha)$ be the set of two-row loop configurations with
these lower and upper connectivities.  For $C\in\mathcal C_L(\beta,\alpha)$,
let $\operatorname{wt}(C)$ be the product of its local $R$-weights.  Let
$\mathbf 1_b(C)$ be one when $e_b$ lies on the through-going interlace, and zero
otherwise.  Equivalently,
\begin{equation}\label{eq:contact-operator-entry}
 \langle\beta|\widehat\rho_b^{(L)}|\alpha\rangle
 =\sum_{C\in\mathcal C_L(\beta,\alpha)}
   \operatorname{wt}(C)\,\mathbf 1_b(C).
\end{equation}
This is a finite connectivity definition, which does not appeal to a scaling
limit.  For generic complex spectral parameters, \eqref{eq:def-P-hat} is a
normalized boundary-contact weight, not a probability.  Its probabilistic
meaning at the homogeneous point is established next.

\subsection{Homogeneous stochastic interpretation}
\label{subsec:stochastic}

Choose $w_0$ so that $w_0^2=-q$ and set $z_1=\cdots=z_L=1$.  Let
\[
 T_L=t_L(w_0;1,\ldots,1)
\]
be the double-row transfer matrix on $\mathrm{LP}_L$.  The entry
$(T_L)_{\beta,\alpha}$ is the total weight of configurations that carry the
lower link pattern $\alpha$ to the upper link pattern $\beta$.

\begin{lemma}
\label{lem:stochastic-ergodic}
The matrix $T_L$ is column stochastic, irreducible, and aperiodic.
Consequently it has a unique stationary probability vector $\pi_L$, and
$T_L^N\mu\to\pi_L$ for every initial probability vector $\mu$ on
$\mathrm{LP}_L$.
\end{lemma}

\begin{proof}
At $z=1$ and $w=w_0$, the two coefficients in a local $R$-tile are
\[
 \frac{[q/w_0]}{[qw_0]}
 \quad\text{and}\quad
 \frac{[1/w_0]}{[qw_0]}.
\]
Since $w_0^2=-q$ and $q^3=1$, one has
\[
 [q/w_0]=[1/w_0]=qw_0,
 \qquad [qw_0]=2qw_0.
\]
Thus each local resolution has weight $1/2$.  At loop weight
$-(q+q^{-1})=1$, contractible loops carry weight one.  Summing over all local
resolutions in a double row therefore gives total mass one for every starting
link pattern.  Thus $T_L$ is column stochastic.

Next, we prove irreducibility directly.  Regard an odd link pattern $\alpha$
as a planar diagram in the Temperley--Lieb algebra.  It joins one defect
endpoint to the $L$ endpoints of a transverse cut.  Let $\alpha^*$ be its vertical
reflection.  For any
$\alpha,\beta\in\mathrm{LP}_L$, the planar tangle
\[
 D_{\beta,\alpha}=\beta\circ\alpha^*
\]
has $L$ lower and $L$ upper endpoints and satisfies
\[
 D_{\beta,\alpha}\,\alpha
 =\beta\circ(\alpha^*\alpha)=\beta,
\]
because $\alpha^*\alpha$ consists of one through interlace and contractible
loops, all of weight one.  Put $D_{\beta,\alpha}$ in Morse position and follow
it from bottom to top.  Insert identity strands and delete contractible loops.
The result is a finite word in the elementary Temperley--Lieb generators
$e_i$ and the identity.  Each $e_i$ is realized by a fixed finite stack of double rows.  Choose the
turning resolution in the two squares next to sites $i,i+1$, and choose the
straight resolution in every other square.  All prescribed local
resolutions have probability $1/2$, so the resulting slab has strictly
positive probability.  Concatenating the slabs realizing the word for
$D_{\beta,\alpha}$ gives $(T_L^N)_{\beta,\alpha}>0$ for some $N$.
Thus $T_L$ is irreducible.

Finally, choosing the straight resolution in every square of a double row
leaves every link pattern unchanged and has positive probability.  Thus every
state has a self-loop, so the chain is aperiodic.  The last assertion is the
standard convergence theorem for a finite irreducible aperiodic Markov chain.
\end{proof}

\begin{proposition}
\label{prop:stochastic-identification}
At the homogeneous point,
\begin{equation}\label{eq:stochastic-identification}
 \pi_L(\alpha)
 =\frac{\psi_\alpha(1,\ldots,1)}{Z_L(1,\ldots,1)},
 \qquad
 \hp_L=\widehat P_L(w_0;1,\ldots,1).
\end{equation}
The second equality is independent of the link patterns imposed at the two
remote ends in the finite-volume approximation of the strip.
\end{proposition}

\begin{proof}
The transfer-matrix eigenvalue equation is
\[
 T_L\Psi_L(1,\ldots,1)=\Psi_L(1,\ldots,1);
\]
see \cite[(3.9)]{IP}.  By Lemma~\ref{lem:stochastic-ergodic}, the eigenspace at
eigenvalue one is one-dimensional.  The sum of the lower components is $Z_L^\downarrow$.
Normalization gives the first identity in
\eqref{eq:stochastic-identification}.  It also shows that the normalized
homogeneous components are nonnegative.  The rotated transfer matrix has the corresponding stationary
law
\[
 \pi_L^\uparrow(\bar\alpha)
 =\frac{\psi_\alpha(1,\ldots,1)}{Z_L^\uparrow(\mathbf1)}.
\]
At the homogeneous point $Z_L^\uparrow(\mathbf1)
=Z_L^\downarrow(\mathbf1)$.

Place the two-row block containing $e_b$ between $N$ ordinary double rows
below and $N$ ordinary double rows above.  Fix arbitrary odd link patterns at
the two remote ends.  The bonds in the lower part, the central block, and the
upper part are independent.  The link patterns induced at the two cuts next
to the central block have laws obtained by applying $T_L^N$
and its rotated counterpart to the chosen end states.  Lemma
\ref{lem:stochastic-ergodic} shows that these laws converge to
$\pi_L^\downarrow$ and $\pi_L^\uparrow$, independently of the end states.
Therefore the limiting probability of the event $e_b\subseteq\ga_L$ is
\[
 \sum_{\alpha,\beta}
 \pi_L^\uparrow(\beta)
 \langle\beta|\widehat\rho_b^{(L)}(w_0;\mathbf1)|\alpha\rangle
 \pi_L^\downarrow(\alpha).
\]
At loop weight one, every closed loop created by gluing has weight one.  The
only nonclosed interlace is the through-going interlace.  In the infinite strip it
becomes $\ga_L$.  Substituting the two stationary laws turns the last
display into the quotient in
\eqref{eq:def-P-hat}.  By
\eqref{eq:contact-operator-entry}, the matrix entry is the conditional
probability that $e_b$ lies on the through-going interlace.  The finite-volume
measures converge to the Bernoulli product measure on the infinite strip.
Hence the limit is
$\Pro_{1/2}(H_L)=\hp_L$.
\end{proof}

\subsection{Local identities for the boundary-contact event}

We first express the restriction to the event $e_b\subseteq\ga_L$ in the
Temperley--Lieb algebra.  Let
$\mathcal T_L^\bullet$ be spanned by tangles in the strip rectangle together
with the distinguished left reflecting edge $e_b$.  Close a basis tangle $D$
with an upper link pattern $\beta$ and a lower link pattern $\alpha$.
Because $L$ is odd, the closure has exactly one interlace joining the two ends of
the strip.  We call it the through-going interlace.  Define the boundary-contact
functional
\begin{equation}\label{eq:tagged-closure-functional}
\begin{aligned}
 \Lambda_b^{\beta,\alpha}(D)
 &=n^{\ell(D)}
  \mathbf 1\{e_b\text{ lies on the through-going interlace}\},\\
 n&=-(q+q^{-1}).
\end{aligned}
\end{equation}
Here $\ell(D)$ is the number of contractible loops created by the closure.
Extend $\Lambda_b^{\beta,\alpha}$ linearly.  The loop-removal relation is
included in \eqref{eq:tagged-closure-functional}.  Hence the functional is
well defined on the diagram algebra.  The edge $e_b$ remains part of the
exterior connectivity.  In particular,
\[
 \langle\beta|\widehat\rho_b^{(L)}|\alpha\rangle
 =\Lambda_b^{\beta,\alpha}(\rho_L),
\]
where $\rho_L$ denotes the full two-row tangle before the
functional is applied.

Let $\varphi_i:V_{L-2}\to V_L$ insert the short Temperley--Lieb cup that
joins sites $i$ and $i+1$.  Let
$\varphi_i^\dagger:V_L\to V_{L-2}$ be the corresponding cap.  Thus
\[
 e_i=\varphi_i\varphi_i^\dagger,
 \qquad
 \varphi_i^\dagger\varphi_i=n\,\mathrm{Id};
\]
at the percolation value $n=1$, the latter factor is one.  The symbol
$\dagger$ denotes this cap operation, not a Hilbert-space adjoint.

Figure~\ref{fig:contact-locality} shows the point used below.  The edge $e_b$
lies outside the disk in which a Yang--Baxter, reflection, or deletion identity
is applied.

\begin{lemma}
\label{lem:contact-local-identities}
As identities of rational operator-valued functions, the operator
$\widehat\rho_b^{(L)}$ satisfies the following relations.

\begin{enumerate}\renewcommand{\labelenumi}{(\roman{enumi})}
\item For $1\le i<L$,
\begin{equation}\label{eq:contact-interlacing}
 \check R_i(z_i/z_{i+1})
 \widehat\rho_b^{(L)}(w;\ldots,z_i,z_{i+1},\ldots)
 =
 \widehat\rho_b^{(L)}(w;\ldots,z_{i+1},z_i,\ldots)
 \check R_i(z_i/z_{i+1}).
\end{equation}

\item The right-end reflection identity holds:
\begin{equation}\label{eq:contact-reflection-right}
 \widehat\rho_b^{(L)}(w;z_1,\ldots,z_{L-1},z_L)
 =\widehat\rho_b^{(L)}(w;z_1,\ldots,z_{L-1},z_L^{-1})
\end{equation}
inside matrix elements between the reflecting qKZ boundary states.

\item At the last pair of columns, which is disjoint from the left reflecting
boundary edge, one has at $z_L=qz_{L-1}$
\begin{equation}\label{eq:contact-operator-deletion}
 \widehat\rho_b^{(L)}(w;z_1,\ldots,z_{L-1},qz_{L-1})
 \circ\varphi_{L-1}
 =
 \varphi_{L-1}\circ
 \widehat\rho_b^{(L-2)}(w;z_1,\ldots,z_{L-2}).
\end{equation}
\end{enumerate}
\end{lemma}

\begin{proof}
Clear the common scalar denominators in the Yang--Baxter, unitarity,
crossing, and reflecting-boundary identities.  Each side is then a finite
linear combination of Temperley--Lieb basis tangles.  The coefficients agree
for every exterior connectivity.
Retain $e_b$ as part of the exterior tangle.  Apply
$\Lambda_b^{\beta,\alpha}$ to both sides.  This functional indicates only
whether $e_b$ lies on the through-going interlace.  It therefore preserves each
identity within every connectivity class.

For interlacing, apply the boundary-contact functional to the usual Yang--Baxter
calculation for the two-row tangle.  When the first two columns are involved,
distinguish configurations according to which exposed interlace is connected to \(e_b\).  The Yang--Baxter identity
remains valid in each resulting connectivity class.  This gives
\eqref{eq:contact-interlacing}.  Applying the same argument to the right
reflection disk, which is disjoint from $e_b$, gives
\eqref{eq:contact-reflection-right}.

For deletion, compose the last two columns with $\varphi_{L-1}$ and set
$z_L=qz_{L-1}$.  The crossing and unitarity identities reduce that block to
$\varphi_{L-1}$ followed by the two-row tangle with $L-2$ columns.  This is the
same reduction as in \cite[Lemma~3.3]{IP}.  The reduction disk is disjoint
from $e_b$.  Moreover, the removed short arc is contractible.  It cannot be
the through-going interlace selected by
\eqref{eq:tagged-closure-functional}.  The event selected by the functional
and the loop factor are unchanged by the reduction.  This proves
\eqref{eq:contact-operator-deletion}.

All statements have been proved after common denominators were cleared.
They consequently hold as rational identities and may be specialized to
$q^3=1$.
\end{proof}

\begin{figure}[ht]
\centering
\begin{tikzpicture}[scale=.82,>=Latex]
  \draw[very thick] (-.25,-.1)--(-.25,2.7);
  \fill[red!70!black] (-.25,1.3) circle (2.2pt);
  \node[anchor=east] at (-.38,1.3) {$e_b$};
  \draw[rounded corners,thick] (0,0) rectangle (5,2.6);
  \draw[thick,->] (.2,.65)--(4.8,.65);
  \draw[thick,<-] (.2,1.95)--(4.8,1.95);
  \draw[thick] (1.65,0)--(1.65,2.6);
  \draw[thick] (3.35,0)--(3.35,2.6);
  \draw[dashed,rounded corners,thick] (1.05,.28) rectangle (3.95,2.32);
  \node[align=center] at (2.5,1.3) {local Yang--Baxter\\or deletion disk};
  \node[align=center] at (2.5,-.42) {$e_b$ is part of the fixed exterior tangle};
\end{tikzpicture}
\caption{Locality behind Lemma~\ref{lem:contact-local-identities}.  The edge
$e_b$ lies outside the disk where the local identity is applied.  Restricting
to configurations in which $e_b$ lies on the through-going interlace is a linear
operation.
It therefore preserves the local identity.}
\label{fig:contact-locality}
\end{figure}
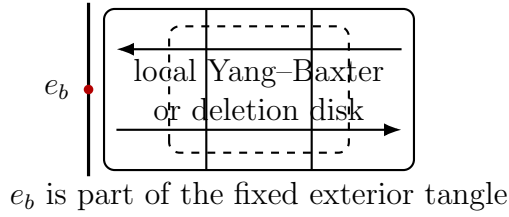

In column $j$, the two $R$-tiles have common denominator
\[
  [qw/z_j][qwz_j]=k(1/w,z_j).
\]
Define
\begin{equation}\label{eq:def-KL}
  K_L(w;\boldsymbol z)=\prod_{j=1}^L k(1/w,z_j)
\end{equation}
and the cleared boundary-contact operator
\[
  \widetilde\rho_L=K_L\widehat\rho_b^{(L)}.
\]
The cleared numerator is
\begin{equation}\label{eq:def-NL}
  N_L(\boldsymbol x;w)
  :=\langle\Psi_L^\uparrow,
  \widetilde\rho_L(w;\boldsymbol z)\Psi_L^\downarrow\rangle.
\end{equation}
Equivalently,
\begin{equation}\label{eq:N-D-P}
  N_L(\boldsymbol x;w)
  =D_L(\boldsymbol x;w)\widehat P_L(w;\boldsymbol z),
  \qquad
  D_L=\mathcal Z_L(\boldsymbol x)K_L(w;\boldsymbol z).
\end{equation}
Unlike $\widehat P_L$, the function $N_L$ is a Laurent polynomial.  The
quantities above have different roles.  The polynomial $Z_L$ is the one-sided
qKZ normalization.  The product $\mathcal Z_L=Z_L^\uparrow Z_L^\downarrow$
is the two-sided normalization.  The factor $K_L$ clears the denominators of
the two-row operator.  Thus $D_L=\mathcal Z_LK_L$ clears all denominators in
$\widehat P_L$, and $N_L=D_L\widehat P_L$ is the Laurent polynomial that will
be interpolated.

\subsection{Coordinatewise Laurent support of the qKZ vector}

For $d\ge0$, let
\[
  \mathcal L_d^{(L)}
  =\operatorname{span}\left\{
  x_1^{a_1}\cdots x_L^{a_L}:
  -d\le a_j\le d\text{ for every }j
  \right\}.
\]
Thus $f\in\mathcal L_d^{(L)}$ means simply that the exponent of each
variable lies between $-d$ and $d$.  The notation $\mathcal L_d^{(L)}$ does
not denote a special family of polynomials.

Let $s_i$ interchange $x_i$ and $x_{i+1}$.  Define
\begin{equation}\label{eq:def-Si}
  \mathcal S_i f
  =\frac{qx_i-q^{-1}x_{i+1}}{x_i-x_{i+1}}(f-s_if).
\end{equation}
After the change from the variables in \cite[(3.8)]{dGP} to
$x_j=z_j^2$, this is the operator $h_i(1)$ in the component form of the qKZ
equation.

\begin{lemma}\label{lem:Si-support}
For every $d\ge0$ and every $i$,
\[
  \mathcal S_i\mathcal L_d^{(L)}\subseteq\mathcal L_d^{(L)}.
\]
\end{lemma}

\begin{proof}
It suffices to use the two variables $X=x_i$, $Y=x_{i+1}$ and a Laurent
monomial $X^aY^b$.  If $a=b$, its image is zero.  If $a>b$, then
\[
  \frac{X^aY^b-X^bY^a}{X-Y}
  =\sum_{r=0}^{a-b-1}X^{a-1-r}Y^{b+r}.
\]
After multiplication by $qX-q^{-1}Y$, every exponent of $X$ and every
exponent of $Y$ lies between $b$ and $a$.  The case $b>a$ follows by
interchanging $X$ and $Y$.  Hence exponents initially in $[-d,d]$ remain in
that interval, and linearity proves the claim.
\end{proof}

We next derive the triangular component relation directly at the root of
unity.  This avoids the generic semisimplicity assumption used in
\cite{dGP}.
Identify odd link patterns with the Dyck paths used in
\cite[Figure~2]{dGP}.  Let $\Omega_L$ be the maximal path and let
$A(\alpha)$ be the number of elementary tiles between $\Omega_L$ and
$\alpha$.

\begin{lemma}
\label{lem:root-component-relation}
Suppose that $\alpha$ has a local maximum at positions $i,i+1$, equivalently
$e_i|\alpha\rangle=|\alpha\rangle$ at loop weight one.  Then
\begin{equation}\label{eq:path-triangular}
 \mathcal S_i\psi_\alpha
 =\sum_{\substack{\beta\ne\alpha\\e_i\beta=\alpha}}\psi_\beta
 =\psi_{\alpha^-}+\sum_{k\ge1}\psi_{\alpha^{+k}}.
\end{equation}
Here $\alpha^-$ is obtained by lowering the local maximum by two units, when
that operation gives an admissible path.  The remaining preimages
$\alpha^{+k}$ are obtained by raising the segment between the maximum and a
neighbouring return at the same height, as in \cite[Figure~2]{dGP}.  In
particular,
\[
 A(\alpha^-)=A(\alpha)+1,
 \qquad
 A(\alpha^{+k})<A(\alpha)<A(\alpha^-).
\]
\end{lemma}

\begin{proof}
Put $u=z_i/z_{i+1}$.  The exchange equation and
\cite[(3.5)]{IP} give, after common denominators have been cleared,
\begin{equation}\label{eq:qkz-exchange-component-vector}
 \left(\frac{[q/u]}{[qu]}\mathbf 1
       -\frac{[u]}{[qu]}e_i\right)\Psi_L
 =s_i\Psi_L .
\end{equation}
Take the coefficient of a link pattern $\alpha$ having a short arc at
$i,i+1$.  At loop weight $-(q+q^{-1})=1$, every link pattern $\beta$ with
$e_i\beta=\alpha$ occurs with coefficient one, including $\beta=\alpha$.
Thus
\[
 \sum_{e_i\beta=\alpha}\psi_\beta
 =\frac{[q/u]}{[u]}\psi_\alpha
  -\frac{[qu]}{[u]}s_i\psi_\alpha .
\]
At $q^3=1$,
\[
 [q/u]-[u]=[qu].
\]
Subtracting the term $\psi_\alpha$ from the preceding sum therefore yields
\[
 \sum_{\substack{\beta\ne\alpha\\e_i\beta=\alpha}}\psi_\beta
 =\frac{[qu]}{[u]}(\psi_\alpha-s_i\psi_\alpha)
 =\frac{qx_i-q^{-1}x_{i+1}}{x_i-x_{i+1}}
   (\psi_\alpha-s_i\psi_\alpha)
 =\mathcal S_i\psi_\alpha.
\]
The planar action of $e_i$ gives the remaining claim.  Its preimages other
than $\alpha$ are $\alpha^-$ and the paths $\alpha^{+k}$ described above.
This is the combinatorial relation shown in \cite[Figure~2]{dGP}.
\end{proof}

In the normalization of \cite[Section~3.5]{IP}, the extremal component has
the following form.  The primitive normalization fixes the nonzero scalar
$c_L$:
\begin{equation}\label{eq:base-component}
  \psi_{\Omega_L}(\boldsymbol z)
  =c_L
  \prod_{1\le i<j\le m}k(z_j,z_i)
  \prod_{m+1\le i<j\le2m-1}k(1/z_i,z_j).
\end{equation}
By \eqref{eq:def-k}, each factor has Laurent exponents in $[-1,1]$ in
every square variable that it contains.  Each $x_j$ occurs in at most $m-1$
factors.  Therefore
\[
  \psi_{\Omega_L}\in\mathcal L_{m-1}^{(L)}.
\]

\begin{proposition}
\label{prop:qkz-coordinate-degree}
For $L=2m-1$ and every link pattern $\alpha$,
\[
  \psi_\alpha\in\mathcal L_{m-1}^{(L)}.
\]
The same bound holds for every component of the rotated vector
$\Psi_L^\uparrow$.
\end{proposition}

\begin{proof}
We induct on $A(\alpha)$.  The case $A(\alpha)=0$ is
\eqref{eq:base-component}.  Let $\beta\ne\Omega_L$ and assume the assertion
for every path of area less than $A(\beta)$.  Choose an addable tile of
$\beta$ and let $\alpha$ be the path obtained by raising the corresponding
valley.  Then $\beta=\alpha^-$ and $A(\alpha)=A(\beta)-1$.  Solving
\eqref{eq:path-triangular} for $\psi_\beta$ gives
\[
  \psi_\beta
  =\mathcal S_i\psi_\alpha-\sum_{k\ge1}\psi_{\alpha^{+k}}.
\]
Every path on the right has area strictly smaller than $A(\beta)$.  The
induction hypothesis and Lemma~\ref{lem:Si-support} therefore put the whole
right-hand side in $\mathcal L_{m-1}^{(L)}$.

The upper vector is obtained by reversing the spectral parameters and
rotating the link pattern.  Hence it has the same coordinatewise support
bound.
\end{proof}

\subsection{Normalization by Laurent interpolation}

\begin{proposition}
\label{prop:normalization}
With the primitive normalization of \cite[Sections~3.4--3.6]{IP},
\begin{equation}\label{eq:normalization-character}
 Z_L(x_1,\ldots,x_L)=\chi_L(x_1,\ldots,x_L)
\end{equation}
for every odd $L$.
\end{proposition}

\begin{proof}
At loop weight one, the all-ones covector on the link-pattern basis satisfies
$\langle\mathbf 1|e_i=\langle\mathbf 1|$.  Hence
\[
 \langle\mathbf 1|\check R_i(u)
 =\frac{[q/u]-[u]}{[qu]}\langle\mathbf 1|
 =\langle\mathbf 1|,
\]
where the last equality is $[q/u]-[u]=[qu]$ at $q^3=1$.  The qKZ exchange
relation therefore implies that $Z_L$ is symmetric in the $x_j$.  The right
reflecting-boundary qKZ relation gives invariance under
$x_L\mapsto x_L^{-1}$. Symmetry then gives inversion invariance in every
variable.

The exact vector recursion \cite[(3.10)]{IP} is
\begin{equation}\label{eq:qkz-vector-deletion}
 \Psi_L\big|_{z_{i+1}=qz_i}
 =(-1)^L\!\prod_{\ell\notin\{i,i+1\}}k(z_i,z_\ell)\,
   \varphi_i\Psi_{L-2}.
\end{equation}
The map $\varphi_i$ sends each link-pattern basis vector to one basis vector
with coefficient one.  Summing the components in
\eqref{eq:qkz-vector-deletion} gives
\begin{equation}\label{eq:Z-deletion}
 Z_L\big|_{z_{i+1}=qz_i}
 =(-1)^L\!\prod_{\ell\notin\{i,i+1\}}k(z_i,z_\ell)\,
   Z_{L-2}.
\end{equation}
The symplectic character $\chi_L$ is symmetric and invariant under
inversion.  It obeys the same deletion relation; see \cite[(3.11)]{IP}.

By Proposition~\ref{prop:qkz-coordinate-degree},
$Z_L\in\mathcal L_{m-1}^{(L)}$.  The largest coordinate of the highest weight
of $\chi_L$ is $m-1$, so Weyl symmetry gives
$\chi_L\in\mathcal L_{m-1}^{(L)}$ as well.  We now induct over odd $L$.
For $L=1$, the primitive normalization gives $Z_1=1=\chi_1$.  Assume the
identity at size $L-2$.  Equations \eqref{eq:Z-deletion} and (3.11) of
\cite{IP} show that $Z_L-\chi_L$, viewed as a Laurent polynomial in $x_L$,
vanishes at $x_L=q^2x_k$ for every $k<L$.  Symmetry and inversion give the
four generic values
\[
 x_L\in\{q^2x_k,q^{-2}x_k,q^2x_k^{-1},q^{-2}x_k^{-1}\},
 \qquad 1\le k<L.
\]
Multiply the difference by $x_L^{m-1}$.  It becomes an ordinary polynomial
of degree at most $2(m-1)=L-1$.  For $L\ge3$, it has at least $4(L-1)$
distinct roots when the other variables are algebraically independent.  It
therefore vanishes identically.  Specialization proves the
identity for all parameter values.
\end{proof}

The character $\chi_L$ is symmetric.  The upper-vector convention
\eqref{eq:upward-vector-definition} and Proposition~\ref{prop:normalization}
therefore give
\begin{equation}\label{eq:two-sided-normalization-character}
 \mathcal Z_L(\boldsymbol x)
 =Z_L(x_L,\ldots,x_1)Z_L(x_1,\ldots,x_L)
 =\chi_L(\boldsymbol x)^2.
\end{equation}

\begin{lemma}
\label{lem:two-sided-vector-deletion}
At $z_L=qz_{L-1}$, put
\begin{equation}\label{eq:deletion-scalar}
 a_L(\boldsymbol z)
 :=(-1)^L\prod_{j=1}^{L-2}k(z_{L-1},z_j).
\end{equation}
Then
\begin{equation}\label{eq:downward-deletion-a}
 |\Psi_L^\downarrow\rangle
 =a_L\,\varphi_{L-1}|\Psi_{L-2}^\downarrow\rangle.
\end{equation}
There is also a nonzero rational scalar $b_L(\boldsymbol z)$ such that
\begin{equation}\label{eq:upward-cap-deletion-b}
 \langle\Psi_L^\uparrow|\varphi_{L-1}
 =b_L\,\langle\Psi_{L-2}^\uparrow|.
\end{equation}
The reduced vectors in both formulas carry
$(z_1,\ldots,z_{L-2})$.
\end{lemma}

\begin{proof}
Equation~\eqref{eq:downward-deletion-a} is exactly the vector
recursion \eqref{eq:qkz-vector-deletion} with $i=L-1$.  For the upper state,
rotate the diagram through $\pi$.  Under the convention
\eqref{eq:upward-vector-definition}, the covector
$\langle\Psi_L^\uparrow|\varphi_{L-1}$ is the transpose of
\begin{equation}\label{eq:rotated-capped-vector}
 \varphi_1^\dagger
 |\Psi_L^\downarrow(z_L,z_{L-1},\ldots,z_1)\rangle.
\end{equation}
Write $u=z_{L-1}$, so the first two parameters in
\eqref{eq:rotated-capped-vector} are $(qu,u)$.  In the normalization
\[
 \check R(s)=\frac{[q/s]}{[qs]}
 \mathbf 1-\frac{[s]}{[qs]}e,
\]
one has, at $q^3=1$,
\[
 \check R(q)=e,
 \qquad\text{since }[1]=0\text{ and }[q^2]=-[q].
\]
The first-pair qKZ exchange relation therefore gives the exact identity
\begin{equation}\label{eq:exact-first-pair-cap-exchange}
 e_1|\Psi_L^\downarrow(qu,u,z_{L-2},\ldots,z_1)\rangle
 =|\Psi_L^\downarrow(u,qu,z_{L-2},\ldots,z_1)\rangle.
\end{equation}
The vector on the right is at the standard deletion specialization.  Combine
\eqref{eq:qkz-vector-deletion} with
$e_1=\varphi_1\varphi_1^\dagger$ and
$\varphi_1^\dagger\varphi_1=n\,\mathrm{Id}=\mathrm{Id}$.  This yields
\[
 \varphi_1^\dagger
 |\Psi_L^\downarrow(qu,u,z_{L-2},\ldots,z_1)\rangle
 =b_L
 |\Psi_{L-2}^\downarrow(z_{L-2},\ldots,z_1)\rangle
\]
for a nonzero rational scalar $b_L$.  This is the rotated-cap calculation from the proof of
\cite[Proposition~4.3]{IP}.  Rotating back proves
\eqref{eq:upward-cap-deletion-b}.
\end{proof}

\begin{corollary}
\label{cor:contact-probability-identities}
The rational function $\widehat P_L(w;\boldsymbol z)$ is symmetric in the
$x_j$, invariant under $x_j\mapsto x_j^{-1}$, and satisfies
\begin{equation}\label{eq:P-deletion-proved}
 \widehat P_L(w;\boldsymbol z)\big|_{z_L=qz_k}
 =\widehat P_{L-2}(w;\widehat z_k,\widehat z_L),
 \qquad 1\le k<L.
\end{equation}
\end{corollary}

\begin{proof}
Insert the qKZ exchange relations for the lower and upper states into
\eqref{eq:def-P-hat}.  Use unitarity of $\check R_i$, then apply
\eqref{eq:contact-interlacing}.  This exchanges $z_i$ and $z_{i+1}$ without
changing the numerator.  The same calculation with the identity in
place of $\widehat\rho_b^{(L)}$ proves the corresponding symmetry of
$\mathcal Z_L$.  The right-reflection relations give invariance of both matrix elements
under $z_L\mapsto z_L^{-1}$.  Permutation symmetry gives inversion invariance
in every variable.  Hence their quotient has all the
stated symmetries.

For deletion, first use permutation symmetry to move $z_k$ to the
$(L-1)$st position.  Write
\[
 M_L=\langle\Psi_L^\uparrow,
       \widehat\rho_b^{(L)}\Psi_L^\downarrow\rangle.
\]
At $z_L=qz_{L-1}$, Lemma~\ref{lem:two-sided-vector-deletion} and
\eqref{eq:contact-operator-deletion} give
\begin{align*}
 M_L
 &=a_L\langle\Psi_L^\uparrow,
       \widehat\rho_b^{(L)}\varphi_{L-1}
       \Psi_{L-2}^\downarrow\rangle\\
 &=a_L\langle\Psi_L^\uparrow,
       \varphi_{L-1}\widehat\rho_b^{(L-2)}
       \Psi_{L-2}^\downarrow\rangle\\
 &=a_Lb_L\langle\Psi_{L-2}^\uparrow,
       \widehat\rho_b^{(L-2)}\Psi_{L-2}^\downarrow\rangle
 =a_Lb_LM_{L-2}.
\end{align*}
The identical computation with the middle operator replaced by the identity
shows
\[
 \mathcal Z_L=a_Lb_L\mathcal Z_{L-2}.
\]
Thus the product of the two scalar factors cancels.  No choice of sign or
cap normalization is needed.  Permuting the remaining
variables back proves \eqref{eq:P-deletion-proved}.
\end{proof}

\subsection{Degree of the cleared numerator}

\begin{lemma}\label{lem:cleared-degree}
Every matrix entry of $\widetilde\rho_L$ belongs to
$\mathcal L_1^{(L)}$ over $\C(w)$.  Consequently,
\[
  N_L(\boldsymbol x;w)\in\mathcal L_L^{(L)}.
\]
\end{lemma}

\begin{proof}
The operator formula is
\[
  \check R(u)=\frac{[q/u]}{[qu]}\,\mathbf1
  -\frac{[u]}{[qu]}\,e.
\]
Multiply the two $R$-tiles in column $j$ by their common denominator
$[qw/z_j][qwz_j]$.  Every local coefficient is then a product of one factor
from
\[
  [qz_j/w],\quad [w/z_j]
\]
and one factor selected from
\[
  [q/(wz_j)],\quad [wz_j].
\]
Such a product contains only the powers $z_j^{-2},1,z_j^2$.  These are the
powers $x_j^{-1},1,x_j$.  Restricting to configurations in which $e_b$ lies
on the through-going interlace only removes terms.  It cannot enlarge the
support.  Thus every entry of $\widetilde\rho_L$ is in
$\mathcal L_1^{(L)}$.

Expanding \eqref{eq:def-NL} in upward and downward link-pattern components,
each summand has coordinate degree at most
\[
  (m-1)+1+(m-1)=2m-1=L
\]
by Proposition~\ref{prop:qkz-coordinate-degree}.  Summation cannot enlarge
Laurent support.
\end{proof}

\subsection{Deletion of the cleared numerator}

The symplectic character obeys
\begin{equation}\label{eq:chi-deletion}
  \chi_M(u_1^2,\ldots,u_M^2)\big|_{u_i=qu_j}
  =(-1)^M\!\prod_{\ell\ne i,j}k(u_j,u_\ell)\,
  \chi_{M-2}(\widehat u_i^2,\widehat u_j^2).
\end{equation}
This is equation~(3.11) of \cite{IP}.  By
Corollary~\ref{cor:contact-probability-identities}, the boundary-contact
weight obeys \eqref{eq:P-deletion-proved}.

Put
\[
  I_k=\{1,\ldots,L\}\setminus\{k,L\},
  \qquad
  A_{L,k}=\prod_{\ell\in I_k}k(z_k,z_\ell).
\]
Then \eqref{eq:chi-deletion} and \eqref{eq:P-deletion-proved} give the full
specialization of the cleared numerator.

\begin{lemma}\label{lem:N-deletion}
At $z_L=qz_k$,
\begin{equation}\label{eq:N-deletion}
  N_L(\boldsymbol x;w)
  =A_{L,k}^2 k(1/w,z_k)k(1/w,qz_k)
  N_{L-2}(\widehat x_k,\widehat x_L;w).
\end{equation}
\end{lemma}

\begin{proof}
Equation~\eqref{eq:two-sided-normalization-character} and
\eqref{eq:chi-deletion} give
\[
  \mathcal Z_L\big|_{z_L=qz_k}=A_{L,k}^2\mathcal Z_{L-2}.
\]
Moreover,
\[
  K_L\big|_{z_L=qz_k}
  =k(1/w,z_k)k(1/w,qz_k)K_{L-2}.
\]
Insert these two identities and \eqref{eq:P-deletion-proved} into
\eqref{eq:N-D-P}.
\end{proof}

\subsection{Candidate numerator and Laurent interpolation}

Define
\begin{equation}\label{eq:N-star}
  N_L^*(\boldsymbol x;w)
  =-[q][w^2]
  \chi_{L+1}(w^2,x_1,\ldots,x_L)
  \chi_{L+1}((q/w)^2,x_1,\ldots,x_L).
\end{equation}
The star is only a label for the proposed expression.  It does not denote
complex conjugation.

\begin{lemma}\label{lem:Nstar-degree}
One has
\[
  N_L^*\in\mathcal L_{L-1}^{(L)}\subsetneq\mathcal L_L^{(L)}.
\]
\end{lemma}

\begin{proof}
For the character $\chi_{L+1}$, the largest part of the highest weight is
$(L-1)/2=m-1$.  Every weight of an irreducible symplectic representation lies
in the convex hull of the Weyl orbit of its highest weight.  Since the Weyl
group permutes coordinates and changes their signs, every exponent of a fixed
character variable lies between $-(m-1)$ and $m-1$.  Multiplying the two
character factors in \eqref{eq:N-star} gives the interval
$[-2(m-1),2(m-1)]=[-(L-1),L-1]$.
\end{proof}

We next verify every factor in the candidate deletion.  At $z_L=qz_k$,
\eqref{eq:chi-deletion} yields
\begin{align}
  \chi_{L+1}(w^2,\boldsymbol x)
  &=(-1)^{L+1}k(z_k,w)A_{L,k}
  \chi_{L-1}(w^2,\widehat x_k,\widehat x_L),
  \label{eq:delete-char-one}\\
  \chi_{L+1}((q/w)^2,\boldsymbol x)
  &=(-1)^{L+1}k(z_k,q/w)A_{L,k}
  \chi_{L-1}((q/w)^2,\widehat x_k,\widehat x_L).
  \label{eq:delete-char-two}
\end{align}
The remaining scalar factors match because
\begin{equation}\label{eq:k-match}
  k(z,w)k(z,q/w)=k(1/w,z)k(1/w,qz).
\end{equation}
Indeed,
\[
\begin{split}
 k(z,w)k(z,q/w)
 &=[qw/z][q/(zw)][q^2/(zw)][w/z],\\
 k(1/w,z)k(1/w,qz)
 &=[qzw][qw/z][q^2zw][w/z],
\end{split}
\]
and $q^3=1$ implies
$[q/(zw)]=-[q^2zw]$ and $[q^2/(zw)]=-[qzw]$.
Multiplying \eqref{eq:delete-char-one} and
\eqref{eq:delete-char-two} therefore gives
\begin{equation}\label{eq:Nstar-deletion}
  N_L^*\big|_{z_L=qz_k}
  =A_{L,k}^2 k(1/w,z_k)k(1/w,qz_k)
  N_{L-2}^*(\widehat x_k,\widehat x_L;w).
\end{equation}
Thus $N_L$ and $N_L^*$ have precisely the same deletion factor.

Both functions are symmetric in the $x_j$ and invariant under
$x_j\mapsto x_j^{-1}$.  For $N_L$, this follows from
Corollary~\ref{cor:contact-probability-identities},
Proposition~\ref{prop:normalization}, and the symmetries of $K_L$.  For
$N_L^*$, it follows from Weyl symmetry.  Assume by induction that the two
numerators agree at size $L-2$.  Then $N_L-N_L^*$ vanishes at
\begin{equation}\label{eq:four-interpolation-points}
  x_L\in
  \{q^2x_k,q^{-2}x_k,q^2x_k^{-1},q^{-2}x_k^{-1}\},
  \qquad 1\le k<L.
\end{equation}

\begin{lemma}\label{lem:Laurent-interpolation}
Let
\[
  F,G\in\C(w,x_1,\ldots,x_{L-1})[u,u^{-1}].
\]
Thus $F$ and $G$ are Laurent polynomials in $u$, with coefficients that are
rational functions of $w,x_1,\ldots,x_{L-1}$.  Suppose that all powers of
$u$ lie in $[-L,L]$.  If they agree at the
$4(L-1)$ generic points $u=q^{\pm2}x_k^{\pm1}$, $1\le k<L$, then
$F=G$ for every odd $L\ge3$.
\end{lemma}

\begin{proof}
The function $u^L(F-G)$ is an ordinary polynomial of degree at most $2L$.
For algebraically independent $x_1,\ldots,x_{L-1}$, the displayed
$4(L-1)$ roots are pairwise distinct.  Since $4(L-1)>2L$ for $L\ge3$, the
polynomial is zero.  The identity then extends to all parameter values by
specialization.
\end{proof}

For $L=1$, one has $Z_1=\chi_1=\chi_2=1$, and the direct four-configuration
calculation in \cite[proof of Proposition~4.6]{IP} gives
\[
\begin{split}
 N_1
 &=[qw/z_1][qz_1w]-[qz_1/w][q/(z_1w)]\\
 &=-[q][w^2]
 =N_1^*.
\end{split}
\]

\begin{theorem}
\label{thm:inhomogeneous}
For every odd $L$,
\begin{equation}\label{eq:inhomogeneous-formula}
  \widehat P_L(w;\boldsymbol z)
  =
  \frac{-[q][w^2]}{\prod_{j=1}^L k(1/w,z_j)}
  \frac{
  \chi_{L+1}(w^2,x_1,\ldots,x_L)
  \chi_{L+1}((q/w)^2,x_1,\ldots,x_L)}
  {\chi_L(x_1,\ldots,x_L)^2}.
\end{equation}
\end{theorem}

\begin{proof}
We prove $N_L=N_L^*$ by induction over odd $L$.  The case $L=1$ was just
checked.  Assume the identity at size $L-2$.  Lemma~\ref{lem:N-deletion} and
\eqref{eq:Nstar-deletion} show that $N_L-N_L^*$ vanishes at $x_L=q^2x_k$ for
every $k<L$.  Symmetry and inversion give the other points in
\eqref{eq:four-interpolation-points}.  Lemmas~\ref{lem:cleared-degree} and
\ref{lem:Nstar-degree} bound the powers of $x_L$ by $[-L,L]$.  Laurent
interpolation therefore yields $N_L=N_L^*$.  Division by $D_L$ gives
\eqref{eq:inhomogeneous-formula} as an identity of rational functions.
\end{proof}

\begin{remark}
The correction to the interpolation in
\cite[proof of Proposition~4.6]{IP} is the following.  We interpolate the
Laurent polynomial $N_L$, not the rational function $\widehat P_L$.
No zero or pole of the denominator is used as an interpolation value.
\end{remark}

\section{Homogeneous specialization and the strip exponent}
\label{sec:homogeneous}

Choose $w_0$ with $w_0^2=-q$ and set $z_1=\cdots=z_L=1$.  The two local
resolutions then both have probability $1/2$, and
\[
  \hp_L=\widehat P_L(w_0;1,\ldots,1).
\]
For $n\ge1$, let
\[
  A(n)=\prod_{j=0}^{n-1}\frac{(3j+1)!}{(n+j)!}.
\]
For $m\ge0$, define
\[
  A_V(2m+1)
  =\prod_{j=0}^{m-1}
  \frac{(3j+2)(6j+3)!(2j+1)!}{(4j+2)!(4j+3)!},
\]
and
\[
  N_8(2m)
  =\prod_{j=0}^{m-1}
  \frac{(3j+1)(6j)!(2j)!}{(4j)!(4j+1)!}.
\]
Empty products are one.

We next state the character identities used at the homogeneous point.  For
completeness, we fix the remaining notation.  If
$\mu_1\ge\cdots\ge\mu_N\ge0$, the Schur polynomial is
\begin{equation}\label{eq:def-schur-polynomial}
 s_\mu(y_1,\ldots,y_N)
 =\frac{\det\!\left(y_i^{\mu_j+N-j}\right)_{i,j=1}^N}
 {\det\!\left(y_i^{N-j}\right)_{i,j=1}^N}.
\end{equation}
If $\nu_1\ge\cdots\ge\nu_N\ge0$, the odd orthogonal character is
\begin{equation}\label{eq:def-odd-orthogonal-character}
 \operatorname{so}^{\mathrm{odd}}_\nu(u_1,\ldots,u_N)
 =\frac{
 \det\!\left(u_i^{\nu_j+N-j+1/2}
             -u_i^{-\nu_j-N+j-1/2}\right)_{i,j=1}^N}
 {\det\!\left(u_i^{N-j+1/2}
             -u_i^{-N+j-1/2}\right)_{i,j=1}^N}.
\end{equation}
The symplectic character $\operatorname{sp}_\lambda$ was defined in
\eqref{eq:def-symplectic-character}.  These determinant quotients are
symmetric polynomials or symmetric Laurent polynomials.  We use only the
identities written below.  We include the determinant reduction because it is
the only step involving the special value $-q$.

\begin{lemma}
\label{lem:character-specializations}
Let $L=2m-1$.  With the character convention of
Section~\ref{sec:qkz},
\begin{align}
 \chi_L(1^L)
 &=3^{(m-1)^2}N_8(L+1),
 \label{eq:chi-homogeneous-N8}\\
 \chi_{L+1}(-q,1^L)
 &=3^{(m-1)^2}\frac{A(L)}{A_V(L)}.
 \label{eq:chi-minus-q}
\end{align}
Here $1^r$ denotes $r$ variables all equal to one.
\end{lemma}

\begin{proof}
We first present the finite dimension formulae.  For a partition
$\lambda=(\lambda_1,\ldots,\lambda_N)$,
\begin{equation}\label{eq:sp-weyl-dimension}
 \operatorname{sp}_{\lambda}(1^N)
 =\prod_{i=1}^N
   \frac{\lambda_i+N-i+1}{N-i+1}
  \prod_{1\le i<j\le N}
   \frac{\lambda_i-\lambda_j+j-i}{j-i}
   \frac{\lambda_i+\lambda_j+2N-i-j+2}
        {2N-i-j+2},
\end{equation}
and, for a partition $\mu$ of length at most $N$,
\begin{equation}\label{eq:schur-weyl-dimension}
 s_\mu(1^N)=
 \prod_{1\le i<j\le N}
 \frac{\mu_i-\mu_j+j-i}{j-i}.
\end{equation}
Substitution of the paired double-staircase parts into
\eqref{eq:sp-weyl-dimension} gives
\begin{align}
 \operatorname{sp}_{\lambda^{(2m-1)}}(1^{2m-1})
 &=3^{(m-1)^2}N_8(2m),
 \label{eq:odd-sp-dimension-evaluation}\\
 \operatorname{sp}_{\lambda^{(2m-2)}}(1^{2m-2})
 &=3^{(m-1)(m-2)}A_V(2m-1).
 \label{eq:even-sp-dimension-evaluation}
\end{align}
Likewise, for
\[
 \mu_L=(L-1,L-1,L-2,L-2,\ldots,1,1,0,0),
\]
formula \eqref{eq:schur-weyl-dimension} gives
\begin{equation}\label{eq:double-staircase-schur-dimension}
 s_{\mu_L}(1^{2L})=3^{L(L-1)/2}A(L).
\end{equation}
These are direct finite-product identities.  The same substitutions appear
in \cite[Section~5.3]{AB}.  Equations
\eqref{eq:odd-sp-dimension-evaluation} and
\eqref{eq:chi-homogeneous-N8} are the same.

It remains to evaluate the special argument $-q$.  Put
\[
 \nu_L=(m-1,m-1,m-2,m-2,\ldots,1,1,0),
 \qquad \widetilde\nu_L=(\nu_L,0)=\lambda^{(L+1)}.
\]
Since $|\nu_L|=m(m-1)$ is even, the double-staircase instance of
\cite[Theorem~1, Eq.~(18)]{AB} is
\begin{equation}\label{eq:AB-odd-orth-factorization}
 s_{\mu_L}(u_1,u_1^{-1},\ldots,u_L,u_L^{-1})
 =\operatorname{so}^{\mathrm{odd}}_{\nu_L}(\boldsymbol u)\,
  \operatorname{so}^{\mathrm{odd}}_{\nu_L}(-\boldsymbol u).
\end{equation}
We now convert the two odd-orthogonal factors into the symplectic factors
needed here.

For $y=x+x^{-1}$, let $\mathcal U_k(y)$ be defined by
\[
 \mathcal U_{-1}=0,\qquad \mathcal U_0=1,
 \qquad \mathcal U_{k+1}=y\mathcal U_k-\mathcal U_{k-1}.
\]
Thus
\[
 \mathcal U_k(x+x^{-1})
 =\frac{x^{k+1}-x^{-k-1}}{x-x^{-1}}.
\]
Set $\mathcal P_k=\mathcal U_k+\mathcal U_{k-1}$. Then
\begin{equation}\label{eq:P-polynomial-odd-orth}
 \mathcal P_k(x+x^{-1})
 =\frac{x^{k+1/2}-x^{-k-1/2}}
        {x^{1/2}-x^{-1/2}}.
\end{equation}
The Weyl bialternant for
$\operatorname{sp}_{\widetilde\nu_L}(-q,u_1,\ldots,u_L)$ has the fixed
variable
\[
 (-q)+(-q)^{-1}=1
\]
and numerator column degrees
\[
 \{3r,3r+1:0\le r\le m-1\}
\]
in the monic polynomial basis $\mathcal U_k$.  Perform the same column operations in the numerator and denominator.
Subtract from each nonconstant column its value at $y=1$ times the constant
column.  Expand along the fixed row, then divide the other rows by $y_i-1$.  The numerator basis which remains is converted to the
$\mathcal P$-basis by the identities
\begin{align}
 \frac{\mathcal U_{3r}(y)-(-1)^r}{y-1}
 &=\sum_{s=1}^{r}(-1)^{r-s}\mathcal P_{3s-1}(y),
 &&1\le r\le m-1,
 \label{eq:U-to-P-even}\\
 \frac{\mathcal U_{3r+1}(y)-(-1)^r}{y-1}
 &=\sum_{s=0}^{r}(-1)^{r-s}\mathcal P_{3s}(y),
 &&0\le r\le m-1.
 \label{eq:U-to-P-odd}
\end{align}
Both identities follow from the recurrence for $\mathcal U_k$.  When the
polynomials are ordered by degree, the two changes of basis are unitriangular.  The resulting
set of numerator degrees is
\[
 \{0,2,3,5,6,\ldots,3m-4,3m-3\},
\]
This is the set of degrees in the odd-orthogonal numerator for $\nu_L$;
see \eqref{eq:P-polynomial-odd-orth}.  In the denominator, both sides are
determinants of monic bases of degrees $0,1,\ldots,L-1$.  They are therefore
the same Vandermonde determinant.
Therefore
\begin{equation}\label{eq:root-conversion-minus-q}
 \operatorname{sp}_{\widetilde\nu_L}(-q,\boldsymbol u)
 =\operatorname{so}^{\mathrm{odd}}_{\nu_L}(\boldsymbol u).
\end{equation}
The substitution $x_i\mapsto-x_i$ multiplies the symplectic character by
$(-1)^{|\widetilde\nu_L|}=1$.  Apply
\eqref{eq:root-conversion-minus-q} to $-\boldsymbol u$.  This gives
\begin{equation}\label{eq:root-conversion-plus-q}
 \operatorname{sp}_{\widetilde\nu_L}(q,\boldsymbol u)
 =\operatorname{so}^{\mathrm{odd}}_{\nu_L}(-\boldsymbol u).
\end{equation}
Combining \eqref{eq:AB-odd-orth-factorization},
\eqref{eq:root-conversion-minus-q}, and
\eqref{eq:root-conversion-plus-q} proves the finite determinant factorization
\begin{equation}\label{eq:double-staircase-factorization}
 s_{\mu_L}(u_1,u_1^{-1},\ldots,u_L,u_L^{-1})
 =\chi_{L+1}(q,\boldsymbol u)\,
  \chi_{L+1}(-q,\boldsymbol u).
\end{equation}

Set $u_1=\cdots=u_L=1$.  By
\eqref{eq:double-staircase-schur-dimension}, the left-hand side equals
$3^{L(L-1)/2}A(L)$.  Now apply \eqref{eq:chi-deletion} to the pair $(q,1)$.  Since $L-1$ is even and
$k(1,1)=[q]^2=-3$,
\[
 \chi_{L+1}(q,1^L)
 =3^{L-1}\chi_{L-1}(1^{L-1})
 =3^{m(m-1)}A_V(L),
\]
where the second equality is
\eqref{eq:even-sp-dimension-evaluation}.  Divide \eqref{eq:double-staircase-factorization} by this value.  The identity
$L(L-1)/2-m(m-1)=(m-1)^2$ then gives \eqref{eq:chi-minus-q}.
\end{proof}

Theorem~\ref{thm:inhomogeneous} and
Lemma~\ref{lem:character-specializations} now give the exact homogeneous
formula.

\begin{theorem}\label{thm:homogeneous}
For every odd $L$,
\begin{equation}\label{eq:homogeneous-formula}
  \hp_L
  =\frac{3}{4^L}
  \frac{A(L)^2}{N_8(L+1)^2A_V(L)^2}.
\end{equation}
\end{theorem}

\begin{proposition}\label{prop:recurrence}
Put $p_m=\hp_{2m-1}$.  Then, for every $m\ge1$,
\[
  \frac{p_{m+1}}{p_m}
  =\left(\frac{6m+1}{6m+2}\right)^2,
\]
and hence
\[
  p_m=\frac34\prod_{j=1}^{m-1}
  \left(\frac{6j+1}{6j+2}\right)^2.
\]
\end{proposition}

\begin{proof}
Take the ratio of \eqref{eq:homogeneous-formula} at widths $2m+1$ and
$2m-1$.  The three product definitions then give
\[
  \left(\frac{p_{m+1}}{p_m}\right)^{1/2}
  =\frac14\frac{(6m-2)(6m+1)}{(3m-1)(3m+1)}
  =\frac{6m+1}{6m+2}.
\]
Since $p_1=3/4$, iteration gives the product.
\end{proof}

\begin{corollary}\label{cor:strip}
There are constants $0<c_0<C_0<\infty$ such that, for every odd $L$,
\[
  c_0L^{-1/3}\le\hp_L\le C_0L^{-1/3}.
\]
\end{corollary}

\begin{proof}
The product recurrence gives
\[
  \log p_m
  =\log\frac34
  +2\sum_{j=1}^{m-1}\log\left(1-\frac1{6j+2}\right)
  =-\frac13\log m+O(1).
\]
Exponentiating and using $L=2m-1$ proves the claim.  Equivalently,
\[
  p_m=\frac34
  \left[
  \frac{\Gamma(4/3)}{\Gamma(7/6)}
  \frac{\Gamma(m+1/6)}{\Gamma(m+1/3)}
  \right]^2.
\]
\end{proof}

\section{From the reflecting edge to an open strip connection}
\label{sec:local-surgery}

We now describe the reflecting edge $e_b$ in square-lattice coordinates.  This makes
the relation between the loop picture and the primal--dual percolation picture
explicit.

Let
\[
 \mathbb L_\diamond
 =\{(x,y)\in\Z^2:x+y\text{ is even}\}
\]
with edges joining $(x,y)$ to $(x+1,y\pm1)$ and their reversals.  This is the
ordinary square lattice rotated by $\pi/4$ and rescaled.  Its dual lattice is
\[
 \mathbb L_\diamond^*
 =\{(x,y)\in\Z^2:x+y\text{ is odd}\},
\]
with the same edge directions.  In this picture, the free boundary is
$x=0$.  The wired boundary is at $x=a_L$, where $a_L\asymp L$.  The comparison
constant depends only on the width assigned to one transfer-matrix column.

Choose the global primal/dual labelling so that the reflecting edge $e_b$
surrounds the primal boundary vertex
\[
 U_b=(0,0).
\]
The two primal edges incident to $U_b$ and entering the strip have midpoints
\[
 m_-=(1/2,-1/2),\qquad m_+=(1/2,1/2).
\]
The curved medial edge $e_b$ joins $m_-$ to $m_+$ around $U_b$ in the
half-plane $x<1/2$.  It is the reflecting edge $e_b$ defined in Section~2.  The dual boundary
vertex $(0,-1)$ belongs to the deterministic dual cluster carried by the free
side.  With this convention, the hull $\ga_L$ is the Dobrushin interface.  The wired
primal cluster lies on its primal side.  The free-boundary dual cluster lies
on its dual side.  Let $W_L$
be the deterministic primal cluster on the opposite wired side, and put
\begin{equation}\label{eq:def-Theta}
 \Theta_L=\Pro(U_b\longleftrightarrow W_L).
\end{equation}

\begin{remark}
Translation by $\tau(x,y)=(x,y-1)$ maps
$\mathbb L_\diamond$ onto $\mathbb L_\diamond^*$ and preserves $x=0$.
Thus the other boundary convention is obtained by applying $\tau$ and
exchanging primal-open with dual-open.  In the usual axis-aligned drawing
of $\Z^2$, this is the familiar $(1/2,1/2)$ primal--dual shift.  We fix
the primal convention above throughout the proof.
\end{remark}

\begin{lemma}
\label{lem:finite-local-surgery}
There are a fixed boundary rectangle $B$, a finite set $\mathcal X$ of
oriented exit bonds, and an integer $M<\infty$.  They do not depend on $L$.
Each bond has the form $f=(u_f,u_f')$, with $u_f\in B$ and $u_f'\notin B$.
For every $f\in\mathcal X$, there is an event $\mathcal T_f$ determined by at
most $M$ bonds inside $B$.  It has the following properties.

\begin{enumerate}\renewcommand{\labelenumi}{(\roman{enumi})}
\item Conditional on the exterior bond configuration,
$\Pro(\mathcal T_f\mid\mathcal F_{B^c})\ge2^{-M}$.
\item Suppose that $f$ is open and that $u_f'$ is connected to $W_L$ without
using a bond whose two endpoints lie in $B$.  Then $\mathcal T_f$ forces
$e_b\subseteq\ga_L$.
\item If $e_b\subseteq\ga_L$, then $U_b\longleftrightarrow W_L$.
\end{enumerate}
\end{lemma}

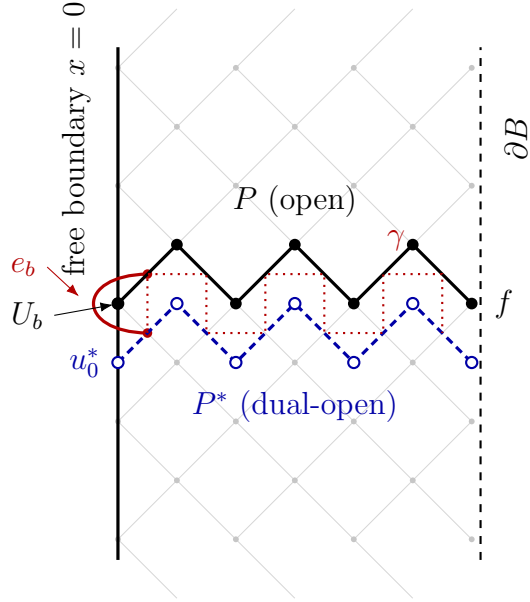
\begin{figure}[ht]
\centering
\begin{tikzpicture}[x=.78cm,y=.78cm,>=Latex]
  \foreach \x in {0,...,6}{
    \foreach \y in {-4,...,4}{
      \pgfmathtruncatemacro{\parity}{mod(\x+\y+20,2)}
      \ifnum\parity=0
        \foreach \dy in {-1,1}{
          \pgfmathtruncatemacro{\xx}{\x+1}
          \pgfmathtruncatemacro{\yy}{\y+\dy}
          \ifnum\xx<7
            \draw[gray!32,very thin] (\x,\y)--(\xx,\yy);
          \fi
        }
        \fill[gray!45] (\x,\y) circle (1.1pt);
      \fi
    }
  }
  \draw[very thick] (0,-4.35)--(0,4.35);
  \draw[dashed,thick] (6.15,-4.35)--(6.15,4.35);
  \node[rotate=90,anchor=south] at (-.28,2.8) {free boundary $x=0$};
  \node[rotate=90,anchor=north] at (6.38,2.8) {$\partial B$};

  \draw[red!70!black,very thick]
     (.5,-.5) .. controls (-.72,-.48) and (-.72,.48) .. (.5,.5);
  \fill[red!70!black] (.5,-.5) circle (1.8pt);
  \fill[red!70!black] (.5,.5) circle (1.8pt);
  \node[red!70!black,anchor=east] (eblab) at (-1.20,.62) {$e_b$};
  \draw[red!70!black,thin,->] (eblab.east)--(-.66,.12);
  \fill (0,0) circle (2.4pt);
  \node[anchor=east] (ublab) at (-1.08,-.20) {$U_b$};
  \draw[thin,->] (ublab.east)--(-.08,-.02);

  \draw[black,very thick]
    (0,0)--(1,1)--(2,0)--(3,1)--(4,0)--(5,1)--(6,0);
  \foreach \p in {(0,0),(1,1),(2,0),(3,1),(4,0),(5,1),(6,0)}
    \fill \p circle (2.2pt);
  \node[anchor=south] at (3,1.28) {$P$ (open)};
  \node[anchor=west] at (6.2,0) {$f$};

  \draw[blue!65!black,very thick,densely dashed]
    (0,-1)--(1,0)--(2,-1)--(3,0)--(4,-1)--(5,0)--(6,-1);
  \foreach \p in {(0,-1),(1,0),(2,-1),(3,0),(4,-1),(5,0),(6,-1)}
    \draw[blue!65!black,fill=white,thick] \p circle (2.1pt);
  \node[blue!65!black,anchor=north] at (3,-1.28)
       {$P^*$ (dual-open)};
  \node[blue!65!black,anchor=east] at (-.12,-1) {$u_0^*$};

  \draw[red!70!black,thick,dotted]
    (.5,-.5)--(.5,.5)--(1.5,.5)--(1.5,-.5)--(2.5,-.5)--(2.5,.5)
    --(3.5,.5)--(3.5,-.5)--(4.5,-.5)--(4.5,.5)--(5.5,.5)--(5.5,-.5);
  \node[red!70!black,anchor=south] at (4.72,.72) {$\ga$};
\end{tikzpicture}
\caption{The local construction in rotated square-lattice coordinates.
Primal vertices have even parity and dual vertices have odd parity.  The
edge $e_b$ joins the two bond midpoints next to $U_b$.  The solid path
$P$ is open.  The bonds crossed by the dashed dual path $P^*$ are closed.  The
red medial interlace lies between these two paths and begins with $e_b$.}
\label{fig:local-surgery}
\end{figure}

\begin{proof}
Take
\[
 B=\{(x,y):0\le x\le6,\ -4\le y\le4\}
\]
and consider the primal bonds of $\mathbb L_\diamond$ that meet this
rectangle.  Bonds crossing $\partial B$ are included in the outside
sigma-field.  Only bonds with both endpoints in $B$ will be changed.

We first treat the straight exit shown in
Figure~\ref{fig:local-surgery}.  Put $\varepsilon_j=0$ for even $j$ and
$\varepsilon_j=1$ for odd $j$, and define
\[
 u_j=(j,\varepsilon_j)\in\mathbb L_\diamond,
 \qquad
 u_j^*=(j,\varepsilon_j-1)\in\mathbb L_\diamond^*,
 \qquad 0\le j\le6.
\]
Force every primal bond
\begin{equation}\label{eq:forced-open-stem}
 P=\bigl\{\{u_j,u_{j+1}\}:0\le j\le5\bigr\}
\end{equation}
to be open.  Force every primal bond crossed by a dual edge in
\begin{equation}\label{eq:forced-closed-fence}
 P^*=\bigl\{\{u_j^*,u_{j+1}^*\}:0\le j\le5\bigr\}
\end{equation}
to be closed.  The two prescribed primal bond sets are disjoint.  The first
dual vertex $u_0^*=(0,-1)$ is on the free boundary, so $P^*$ is attached to
the deterministic free-side dual cluster.

The first open bond of $P$ is
$\{(0,0),(1,1)\}$ and the first closed bond crossed by $P^*$ is
$\{(0,0),(1,-1)\}$.  Their midpoints are $m_+$ and $m_-$, respectively.
The two local medial resolutions now show that the interlace between $P$ and
$P^*$ contains the exterior edge joining $m_-$ to $m_+$.  This edge is $e_b$.
The two paths are separated by one lattice spacing. Hence the interface is forced to pass through the region between them and leaves \(B\) between the two terminal half-edges.

Let $\mathcal X$ be the set of oriented primal bonds crossing the three
non-free sides of $B$.  Orient them from $B$ to its exterior, and include them
in $\mathcal F_{B^c}$.  Fix $f=(u_f,u_f')\in\mathcal X$.  Choose a simple
primal path $P_f$ from $U_b$ to $u_f$.  It begins with
$\{(0,0),(1,1)\}$ and otherwise avoids the free boundary.  Take the union of the diamond faces next to $P_f$ on its free-boundary
side.  One component of the dual boundary is a nearest-neighbour dual path
$P_f^*$.  It begins with the dual edge crossing
$\{(0,0),(1,-1)\}$.  After adding one fixed layer to \(B\), its terminal edge can be chosen adjacent to the side of \(f\) through which the exploration path exits. Thus, \(P_f\) and \(P_f^*\) are disjoint and bound a region of width one extending from the reflecting edge to this exit side.
This is a finite planar construction because $\mathcal X$ is finite.

Declare the bonds of $P_f$ open and the bonds crossed by $P_f^*$ closed.
Also prescribe the finitely many bonds mear the terminal point.  Choose them so
that the medial interlace leaves between $P_f$ and $P_f^*$.  Along the interlace, there is an open primal path on one side and a
dual-open path on the other. It cannot cross either path or leave the region between them.
This defines $\mathcal T_f$.  Since $\mathcal X$ is finite, the number of
prescribed bonds is at most a fixed integer $M$.  Independence gives (i).

Assume the condition in (ii).  On $\mathcal T_f$, the primal path belongs
to the wired primal cluster.  The dual path belongs to the fixed dual cluster
on the free side.  A loop component
separating clusters attached to the two opposite side boundaries cannot be a
contractible loop.  The separating component just identified is therefore
the unique infinite hull $\ga_L$, and it contains $e_b$.  This proves
(ii).

Finally, the convention fixes the two sides of $e_b$.  The primal side
contains $U_b$.  The dual side contains the free-boundary dual cluster.  If
$e_b\subseteq\ga_L$, its primal-side cluster is the opposite wired
cluster.  Hence $U_b\longleftrightarrow W_L$, proving (iii).
\end{proof}

\begin{lemma}
\label{lem:localcomparison}
There exist constants $0<c_1<C_1<\infty$, independent of odd $L$, such that
\begin{equation}
\label{eq:localcomparison}
  c_1\Theta_L
  \le
  \hp_L
  \le
  C_1\Theta_L.
\end{equation}
\end{lemma}

\begin{proof}
Part (iii) of Lemma~\ref{lem:finite-local-surgery} gives the inclusion
$H_L\subseteq\{U_b\longleftrightarrow W_L\}$ and hence
$\hp_L\le\Theta_L$.

For the reverse inequality, let $\mathcal F_{B^c}$ be generated by the bonds
that do not have both endpoints in $B$.  Write
$u\stackrel{B^c}{\longleftrightarrow}W_L$ when such bonds contain an open path
from $u$ to $W_L$.  Define
\begin{equation}\label{eq:E-out}
 E_{\mathrm{out}}
 =\bigcup_{f=(u_f,u_f')\in\mathcal X}
   \bigl\{f\text{ is open},\
   u_f'\stackrel{B^c}{\longleftrightarrow}W_L\bigr\}.
\end{equation}
When $L$ is larger than the diameter of $B$, every path from $U_b$ to
$W_L$ has a last exit from $B$.  The remaining part of the path lies outside
$B$.
Consequently
\[
 \{U_b\longleftrightarrow W_L\}\subseteq E_{\mathrm{out}},
 \qquad
 \Pro(E_{\mathrm{out}})\ge\Theta_L.
\]
On $E_{\mathrm{out}}$, choose the first connected exit $f$ in a fixed
deterministic ordering.  This choice is $\mathcal F_{B^c}$-measurable.
Lemma~\ref{lem:finite-local-surgery} and independence of the bonds in $B$
then give
\[
 \Pro(H_L\mid\mathcal F_{B^c})
 \ge 2^{-M}\mathbf1_{E_{\mathrm{out}}}.
\]
Taking expectations yields
\[
 \hp_L\ge2^{-M}\Pro(E_{\mathrm{out}})
 \ge2^{-M}\Theta_L.
\]
The finitely many smaller odd widths are absorbed by decreasing the constant.
\end{proof}

\section{From the strip connection to the half-plane one-arm estimate}

\begin{lemma}
\label{lem:striphalfplane}
There exist constants $0<c_2<C_2<\infty$ such that, for every odd $L\ge3$,
\begin{equation}
\label{eq:striphalfplane}
  c_2\onearm(L)
  \le
  \Theta_L
  \le
  C_2\onearm(L).
\end{equation}
\end{lemma}

\begin{proof}
We use $\ell^\infty$ boxes centred at $U_b$.  Changing the norm changes the
probability by at most a universal factor.  The same is true if the initial
boundary vertex is moved by a bounded distance.

We first prove the upper bound.  If
\[
  U_b\longleftrightarrow W_L,
\]
then an open path starting from $U_b$ reaches the opposite side of a strip
of width $L$.  In particular, for all sufficiently large $L$, this path
leaves $B_{L/2}$.  Hence
\[
  \Theta_L
  \le
  C\,\onearm(L/2).
\]
By fixed-ratio arm extension, which follows from the RSW theorem,
\[
  \onearm(L/2)\le C'\onearm(L).
\]
Therefore
\[
  \Theta_L\le C_2\onearm(L).
\]

For the converse, put
\[
  R=\left\lfloor\frac{L}{8}\right\rfloor.
\]
For all sufficiently large $L$, consider the following three increasing
events:
\begin{enumerate}\renewcommand{\labelenumi}{(\alph{enumi})}
\item there is an open arm from the fixed boundary neighbourhood of $U_b$
      to $\partial B_{2R}$;

\item there is an open half-circuit in the half-annulus
      \[
        B_{2R}\setminus B_R
      \]
      joining the two portions of the boundary line;

\item there is an open crossing from a deterministic segment contained in
      $B_R$ to the opposite wired side $W_L$.
      Equivalently, one may require crossings of a fixed chain of
      overlapping rectangles of uniformly bounded aspect ratio joining
      $B_R$ to $W_L$.
\end{enumerate}

We claim that on the intersection of these three events,
\[
  U_b\longleftrightarrow W_L.
\]
Indeed, the arm in (a) starts inside $B_R$ and reaches
$\partial B_{2R}$.  Hence it crosses the half-annulus
$B_{2R}\setminus B_R$.  The half-circuit in (b), together with the
corresponding portion of the boundary line, separates the inner part of
the half-disk from the exterior of $B_{2R}$.  By planarity, the arm in
(a) must therefore intersect the half-circuit in (b).

The open path in (c) also starts inside $B_R$.  It reaches the wired side,
which lies outside $B_{2R}$.  Hence it crosses the same half-annulus and must
meet the half-circuit in (b).  Consequently all
three open pieces belong to the same open cluster, and this cluster
connects $U_b$ to $W_L$.

Therefore
\[
  \Theta_L
  \ge
  \Pro\bigl((a)\cap(b)\cap(c)\bigr).
\]
All three events are increasing, so the FKG inequality gives
\[
  \Pro\bigl((a)\cap(b)\cap(c)\bigr)
  \ge
  \Pro((a))\,\Pro((b))\,\Pro((c)).
\]

The event in (a) has probability comparable to $\onearm(2R)$, up to a
multiplicative constant independent of \(R\)::
\[
  \Pro((a))\ge c_0\onearm(2R).
\]
The RSW theorem gives a universal constant $c_1>0$ such that
\[
  \Pro((b))\ge c_1.
\]
The ratio $L/R$ is bounded.  Thus the event in (c) follows from crossings
of a fixed number of rectangles with bounded aspect ratio.  The RSW theorem and the FKG inequality therefore give
\[
  \Pro((c))\ge c_2'>0.
\]
It follows that
\[
  \Theta_L
  \ge
  c\,\onearm(2R).
\]
Since $2R\le L$ and the one-arm probability is decreasing in its outer
radius,
\[
  \onearm(2R)\ge\onearm(L).
\]
Hence
\[
  \Theta_L\ge c_2\onearm(L).
\]

The finitely many smaller odd values of $L$ are absorbed by decreasing
$c_2$ and increasing $C_2$ if necessary.  This proves
\eqref{eq:striphalfplane}.
\end{proof}

\begin{lemma}
\label{lem:onearm-quasimultiplicativity}
There is $C_{\mathrm{qm}}<\infty$ such that, for all integers
$1\le r\le s\le R$,
\begin{equation}\label{eq:onearm-quasimultiplicativity}
 C_{\mathrm{qm}}^{-1}\onearm(r,s)\onearm(s,R)
 \le \onearm(r,R)
 \le C_{\mathrm{qm}}\onearm(r,s)\onearm(s,R).
\end{equation}
\end{lemma}

\begin{proof}
First assume $4r\le s\le R/4$.  The event $A_1^+(r,R)$ gives an arm in
each of the two separated annuli $(r,s)$ and $(s,R)$.  Assign a fixed-width
neighbourhood of radius $s$ to one of these annuli.  Independence then gives
the upper bound in \eqref{eq:onearm-quasimultiplicativity}, with a fixed-ratio
change of the middle scale.  Those changes cost only a universal factor by the
RSW theroem.

For the lower bound, consider three increasing events.  Let
$E_{\mathrm{in}}$ be the event that an open arm joins $\partial B_r$ to
$\partial B_{2s}$.  Let $E_{\mathrm{circ}}$ be the event that an open
half-circuit crosses $B_{2s}\setminus B_{s/2}$ and joins the two portions of
the boundary line.  Let $E_{\mathrm{out}}$ be the event that an open arm joins
$\partial B_{s/2}$ to $\partial B_R$.

On $E_{\mathrm{in}}$, the arm starts inside $B_{s/2}$ and reaches
$\partial B_{2s}$.  It therefore crosses the half-annulus
$B_{2s}\setminus B_{s/2}$.  Planarity forces it to meet the half-circuit on
$E_{\mathrm{circ}}$.

On $E_{\mathrm{out}}$, the arm starts at $\partial B_{s/2}$ and reaches
$\partial B_R$, where $R>2s$.  It also crosses the same half-annulus and must
meet the half-circuit.

Therefore, on
\[
E_{\mathrm{in}}\cap E_{\mathrm{circ}}\cap E_{\mathrm{out}},
\]
the two arms are connected through the half-circuit and hence belong to
a single open cluster containing an arm from $\partial B_r$ to
$\partial B_R$.  Thus
\[
E_{\mathrm{in}}\cap E_{\mathrm{circ}}\cap E_{\mathrm{out}}
\subseteq
A_1^+(r,R),
\]
and consequently
\[
\onearm(r,R)
\ge
\Pro\bigl(
E_{\mathrm{in}}\cap E_{\mathrm{circ}}\cap E_{\mathrm{out}}
\bigr).
\]

Since all three events are increasing, the FKG inequality gives
\[
\begin{aligned}
\onearm(r,R)
&\ge
\Pro(E_{\mathrm{in}})
\Pro(E_{\mathrm{circ}})
\Pro(E_{\mathrm{out}}).
\end{aligned}
\]
The RSW thoerem yields a universal constant $c>0$ such that
\[
\Pro(E_{\mathrm{circ}})\ge c,
\]
and therefore
\[
\onearm(r,R)
\ge
c\,\onearm(r,2s)\onearm(s/2,R).
\]

Finally, fixed-ratio arm extension gives
\[
\onearm(r,2s)\asymp\onearm(r,s),
\qquad
\onearm(s/2,R)\asymp\onearm(s,R).
\]
Hence
\[
\onearm(r,R)
\ge
C_{\mathrm{qm}}^{-1}
\onearm(r,s)\onearm(s,R).
\]

If $s<4r$ or $R<4s$, the  RSW theorem bounds the missing factor above and
below by positive constants.  Thus the same estimate holds for every
$r\le s\le R$.  This is also the
one-arm case of \cite[Section~6.2]{DMT}.
\end{proof}

We can now prove the main result.

\begin{theorem}
\label{thm:onearm}
There exist constants $0<c<C<\infty$ such that, for every $R\ge2$,
\begin{equation}
\label{eq:onearmscale1}
  cR^{-1/3}
  \le
  \onearm(R)
  \le
  CR^{-1/3}.
\end{equation}
More generally, for all integers $1\le r<R$,
\begin{equation}
\label{eq:onearmgeneral}
  c\left(\frac rR\right)^{1/3}
  \le
  \onearm(r,R)
  \le
  C\left(\frac rR\right)^{1/3}.
\end{equation}
\end{theorem}

\begin{proof}
For odd $L$, Lemmas~\ref{lem:localcomparison} and
\ref{lem:striphalfplane}, together with Corollary~\ref{cor:strip}, give
\[
  \onearm(L)
  \asymp
  \Theta_L
  \asymp
  \hp_L
  \asymp
  L^{-1/3}.
\]
Choosing a neighbouring odd integer and changing the radius by at most a
bounded amount proves \eqref{eq:onearmscale1} for every $R$.

Lemma~\ref{lem:onearm-quasimultiplicativity} gives
\[
 \onearm(R)\asymp\onearm(r)\onearm(r,R).
\]
Dividing the estimates in \eqref{eq:onearmscale1} at scales $R$ and $r$
proves \eqref{eq:onearmgeneral}.
\end{proof}


\begin{thebibliography}{99}
\setlength{\itemsep}{2pt}

\bibitem{AB}
A.~Ayyer and R.~E. Behrend,
\emph{Factorization theorems for classical group characters, with applications
to alternating sign matrices and plane partitions},
J. Combin. Theory Ser. A \textbf{165} (2019), 78--105.

\bibitem{dGP}
J.~de Gier and P.~Pyatov,
\emph{Factorised solutions of Temperley--Lieb $q$KZ equations on a segment},
Adv. Theor. Math. Phys. \textbf{14} (2010), 795--877.

\bibitem{DMT}
H.~Duminil-Copin, I.~Manolescu, and V.~Tassion,
\emph{Planar random-cluster model: fractal properties of the critical phase},
Probab. Theory Related Fields \textbf{181} (2021), 401--449.

\bibitem{IP}
Y.~Ikhlef and A.~Ponsaing,
\emph{Finite-size left-passage probability in percolation},
J. Stat. Phys. \textbf{149} (2012), 10--36.

\end{thebibliography}
\end{document}